\documentclass[12pt, a4paper]{amsart} 
\usepackage[T2A]{fontenc}
\usepackage[dvips]{graphicx}
\usepackage{geometry}
\usepackage{amsmath, amssymb, amsfonts, amsthm}
\usepackage{enumitem}
\usepackage{graphicx}
\usepackage{tikz}
\usepackage{hyperref}
 \usepackage{dutchcal}
\usepackage{bbm}

\usepackage{comment}
\usepackage[utf8]{inputenc}
\usepackage[T2A]{fontenc}
\usepackage[english]{babel}
\usepackage{eulervm}

\newtheorem{lemma}{Lemma}[section]
\newtheorem{theorem1}{Theorem}
\newtheorem{theorem}{Theorem}
\newtheorem{fact}{Conjecture}

 \newtheorem{rem}{Remark}
\newenvironment{proof}{\noindent{\bf Proof:}}{$\hfill \Box$ \vspace{10pt}}  
\def\Xint#1{\mathchoice
{\XXint\displaystyle\textstyle{#1}}%
{\XXint\textstyle\scriptstyle{#1}}%
{\XXint\scriptstyle\scriptscriptstyle{#1}}%
{\XXint\scriptscriptstyle\scriptscriptstyle{#1}}%
\!\int}
\def\XXint#1#2#3{{\setbox0=\hbox{$#1{#2#3}{\int}$ }
\vcenter{\hbox{$#2#3$ }}\kern-.6\wd0}}

\def\dashint{\Xint-}

\newcommand{\R}{\mathbb{R}}

\newcommand\T{\mathbb{T}}

\begin{document}
\title{The property $\log(f)\in \mathrm{BMO}(\R^n)$ in terms of Riesz transforms.}
\date{10 January 2013}

\author[Ioann Vasilyev]{Ioann Vasilyev}
\email{ioann.vasilyev@cyu.fr, milavas@mail.ru}
\thanks{Statements and Declarations: The author declares no competing interests.}

\begin{abstract}
The condition mentioned in the title is equivalent to the representability of $f$ as the quotient
$f= v_1 /v_2$ , where $v_1$ and $v_2$ obey the inequalities $|R_j v_i| \leq cv_i, i = 1,2, j=1, \ldots ,n$. Here, $R_1, \ldots, R_n$ are the Riesz transformations.
\end{abstract}

\maketitle

\section{Introduction}
In the paper [2, p. 695] the following lemma has been proved.

\begin{lemma}
\label{lem1}
Let $ f>0$ be a measurable function on $\T$ (where $\T$ is a unit circle). Then $\log f\in BMO(\T)$ if and only if there exists $c>1, 0<\rho<1$ and a function $w>0$ such that $\frac{f}{c}\leq w \leq c\*f$
and $|H(w^\rho)|\leq cw^\rho$, where $H$ is a Hilbert transform on $\T$ (namely, $H(f)(x)=\frac{1}{2\pi}\; \dashint \limits_\T \frac{f(t)}{\tan(x-t)}\,dt$).
\end{lemma}

Lemma~\ref{lem1} relates only to the unit circle $\T$. In this regard a question arises about creation of a similar result for space $\R^n$. To our mind, a natural formulation of Lemma~\ref{lem1} for the multidimensional case is the following conjecture.

\begin{fact}
\label{conj1}
Let $f>0$ be a measurable function on $\mathbb R^n$. Then $\log f\in BMO(\R^n)$ if and only if there exist $c>1, 0<\rho<1$ and a function $w$, such that $\frac{f}{c}<w<c\*f$
and $|R_j(w^\rho)|<cw^\rho$ for all $j$ from $1$ to $n$, where $R_j$ is an $j$-th Riesz transform in $\R^n$ (i.e. $R_j(f)(x)=c_n \dashint \limits_{\R^n} \frac{x_j-t_j}{|x-t|^{n+1}}f(t)\,dt$, where $c_n=\frac{\Gamma[\frac{n+1}{2}]}{\pi^{\frac{n+1}{2}}}$).
\end{fact}

I succeeded to prove only sufficiency in above mentioned expansion of Lemma~\ref{lem1} in case of $BMO(\R^n)$. Nevertheless a very similar to Conjecture~\ref{conj1} statement was proved in full generality. Precisely, the following theorem is valid.
\begin{theorem}
\label{thm1}
Let $f>0$ be a measurable function on $\R^n$ function. Then $\log f\in BMO(\R^n)$ if and only if there exists positive functions $g_1\in L^2(\R^n), g_2\in L^2(\R^n)$ and $\alpha \in \R$ such that $f=(\frac{g_1}{g_2})^\alpha$ and $|R_jg_i|<cg_i$ for all $j$ from $1$ to $n$ and $i\in \{1,2\}$.
\end{theorem}

This paper is entirely devoted to the proof of Theorem~\ref{thm1}.

 \begin{rem}
An interest for the multidimensional equivalents of Lemma~\ref{lem1} is based on a strong application of that lemma in interpolation theory of one-dimensional Hardy spaces (see [2]). But it is unclear yet if Theorem~\ref{thm1}  proved in this paper has any interpolation applications.
\end{rem}

\section{Two main lemmas}
Let us formulate and prove two main lemmas that are needed later. The first of that lemmas will be used in the proof of sufficiency in Theorem~\ref{thm1}, the second one is exactly necessity in Theorem~\ref{thm1}.  

\begin{lemma}["Sufficiency"]
Let function $f>0$ belong to the space $L^p(\R^n)$ for some  $p>1$. If $|R_jf|<cf$ for all $j$ from $1$ to $n$, then  $\log f\in BMO(\R^n)$.
\end{lemma}
\begin{proof}
Let us introduce an auxiliary system of conjugate harmonic functions: 
$$F=(U_f,U_{R_1f},...,U_{R_nf})\;$$
where 
$$U_g(x,t)=\int\limits_{\R^n}g(x-y)\frac{tdy}{(t^2+|y|^2)^\frac{n+1}{2}}$$ is the Poisson integral of the function $g$.

Since $f \in L^p(\R^n)$, the function 
$$|F|(x,t)= \left(\sum_{i=1}^n U_{R_jf}^2(x,t)+U_f^2 \right) ^{1/2}$$ 
is subharmonic raised to a power $\varepsilon$ for every  $\varepsilon$ strictly greater than $\frac{n-1}{n}$  (see [1, p.286]). 

Next, for $t_0>0$ we introduce three auxiliary notations 
$$G_{t_0}(x):= |F|^\varepsilon(x,t_0), \; H_{t_0}:=U_{G_{t_0}}, \; U_{t_0}(x,t):=|F|^\varepsilon(x,t_0+t).$$ 
Observe that  $H_{t_0}(x,t)$ is a harmonic function in  $\R_{+}^{n+1}$ and that at the point $t=0$ it is equal to $G_{t_0}(x)$. On top of that, the function $U_{t_0}(x,t)$ is subharmonic in $\R_+^{n+1}$ and holds $$U_{t_0}(x,t)|_{t=0}=G_{t_0}(x).$$

Consider the following auxiliary function 
$$W_{t_0}(x,t):=U_{t_0}(x,t)-H_{t_0}(x,t).$$
Note that the function $W_{t_0}(x,t)$ being continuous at $(x,t)\in \R^{n+1}_{+}$ decreases to zero at infinity and is equal to zero at the boundary.
Hence the maximum principle holds for $R_+^{n+1}$  and subharmonic function $W_{t_0}(x,t)$ (see [3, p. 64]).
From this it follows that because $H_{t_0}(x,t)$ majorizes $U_{t_0}(x,t)$ on the boundary, $H_{t_0}(x,t)$ majorizes $U_{t_0}$ everywhere in $\R_{+}^{n+1}$, namely for any point $(x,t) \in \R_{+}^{n+1}$ holds $U_{t_0}(x,t) \leq H_{t_0}(x,t)$.  

Hence,
\begin{equation*}
|F|^\varepsilon(x,t_0+t)
\leq\int\limits_{\R^n}G_{t_0}(x-y)\frac{tdy}{(t^2+|y|^2)^\frac{n+1}{2}}
=\int\limits_{\R^n}|F|^\varepsilon(x-y,t_0)\frac{tdy}{(t^2+|y|^2)^\frac{n+1}{2}}.
\end{equation*}
Further, passing to limit when $t_0\rightarrow 0$, we have that:
\begin{multline*}
|F|^{\varepsilon}(x,t)\leq \inf_{t_0>0} \int \limits_{\R^n}|F|^\varepsilon(x-y,t_0)\frac{tdy}{(t^2+|y|^2)^\frac{n+1}{2}}\\
\leq\int\limits_{\R^n}\left(\sum_{i=1}^n((R_i (f))(x-y))^2+(f(x-y))^2\right)^\frac{\varepsilon}{2}\frac{tdy}{(t^2+|y|^2)^\frac{n+1}{2}}.
\end{multline*}

In the last inequality above we have used the Fatou theorem. From here we have: 
$$\left(\left({U_{f}(x,t)}\right)^{2}+\sum_{j=1}^n {\left(U_{R_{j}f}(x,t) \right)^{2}}\right)^{\frac{\varepsilon}{2}}\leq C_1\int\limits_{\R^n}f^\varepsilon(y)\frac{tdy}{(t^2+|x-y|^2)^\frac{n+1}{2}}. $$
We omit the second term in the left hand side of the last inequality above since this term is positive:
$$\left (\int \limits_{\R^n} f(y)\frac{tdy}{(t^2+|x-y|^2)^\frac{n+1}{2}}\right)^{\varepsilon}\leq C_1\int \limits_{\R^n}f^{\varepsilon}(y)\frac{tdt}{(t^2+|x-y|^2)^{\frac{n+1}{2}}}.$$
for all $(x,t)\in \R_+^{n+1}$. 

Introduce three auxiliary notations, for $q>1$ we pose $\varepsilon:=1/q, f_1:=f^\varepsilon=$ ; hence it holds that $f=f_1^q$.
Then we have that: 
\begin{equation} 
	\label{eq1}
\int \limits_{\R^n} f_1^{q}(y)\frac{tdy}{(t^2+|x-y|^2)^\frac{n+1}{2}}\leq C_1\left(\int \limits_{\R^n}f_1(y)\frac{tdt}{(t^2+|x-y|^2)^{\frac{n+1}{2}}}\right)^q.
\end{equation}

We further fix a cube $Q\subset \R^n$. Let  $Q^N$ be the cube centered at the center of $Q$ whose facets are parallel to those of $Q$ and whose edge length is $N$ times longer than that of an edge of $Q$. We choose $t$ in~\eqref{eq1} equal to the edge length of $Q$ (namely, $t$ is equal to $l$) and choose $x$ to be the center of $Q$. 

We shall estimate the left and the right sides of the inequality~\eqref{eq1} separately. We begin with the left hand side.
\begin{multline}
\label{eq2}
\int \limits_{\R^n} f_1^{q}(y)\frac{ldy}{(l^2+|x-y|^2)^\frac{n+1}{2}}=\\=\int \limits_{Q^N} f_1^{q}(y)\frac{ldy}{(l^2+|x-y|^2)^\frac{n+1}{2}}+\int \limits_{{\R^n} \textbackslash{Q^N}} f_1^{q}(y)\frac{ldy}{(l^2+|x-y|^2)^\frac{n+1}{2}}\\
 \geq \frac{1}{(N^2+1)^\frac{n+1}{2}}\int\limits_{Q^N}f_1^q(y)dy\cdot \frac{1}{|Q|}+\int \limits_{{\R^n} \textbackslash{Q^N}} f_1^q(y)\frac{ldy}{(l^2+|x-y|^2)^\frac{n+1}{2}}.
\end{multline}
 Next we treat the right hand side.
\begin{multline}
\label{eq3}
 C_1 \left(\int \limits_{\R^n} f_1(y)\frac{ldy}{(l^2+|x-y|^2)^\frac{n+1}{2}}\right)^q=\\
=C_1\cdot \left(\int \limits_{Q^N} f_1(y)\frac{ldy}{(l^2+|x-y|^2)^\frac{n+1}{2}}+\int \limits_{{R^n} \textbackslash{Q^N}} f_1(y)\frac{ldy}{(l^2+|x-y|^2)^\frac{n+1}{2}}\right)^q\leq\\ 
\leq C_1\cdot 2^{q-1}\cdot \left(\int \limits_{Q^N} f_1(y)\frac{ldy}{(l^2+|x-y|^2)^\frac{n+1}{2}}\right)^q+\\
+C_1\cdot 2^{q-1}\cdot \left(\int \limits_{{\R^n} \textbackslash{Q^N}} f_1(y)\frac{ldy}{(l^2+|x-y|^2)^\frac{n+1}{2}}\right)^q\leq \cdots .
\end{multline} 

After denoting the first term in~\eqref{eq3} by $A$ we can write as follows.
$$\cdots \leq A+ C_1  2^{q-1} \|\varphi_N\|_{L^1 (\R^n)}^q\cdot \left(\int \limits_{R^n}f_1(y) \cdot \frac{\varphi_n(y) dy}{\|\varphi_n\|_{L^1 (\R^n)}}\right)^q \leq \cdots$$ 
where $\varphi_N$ is a function  defined as follows:
$$ \varphi_N(y)=
 \begin{cases}
\frac{l}{(l^2+|x-y|^2)\frac{n+1}{2}}, &\text{if } y \in \R^n \textbackslash{Q^N};\\
 0,&\text{if }  y\in Q^N.\\
 \end{cases}
 $$

We continue the estimates using Hölder's inequality: 
$$\cdots\leq A+C_1 2^{q-1}  \|\varphi_N\|_{L^1 (\R^n)}^q  \cdot  \int \limits_{\R^n}f_1^q(y)  \frac{\varphi_n(y) dy}{\|\varphi_n\|_{L^1 (\R^n)}} = 
$$
$$=A +C_1  2^{q-1}  \|\varphi_n\|_{L^1 (\R^n)}^{q-1} \cdot \int\limits_{{\R^n} \textbackslash{Q^N}} f_1^q(y)\frac{ldy}{(l^2+|x-y|^2)\frac{n+1}{2}}$$ 

We are now in position to choose $N$. We shall need from that number the following two things:
\\
(a). it is necessary that $C_1 2^{q-1} \|\varphi\|_{L^1(\R^n)}^{q-1} $ is less than $1$ ;
\\
(b). it is necessary that $N$ is independent from $Q$.

Let us now estimate the norm $\|\varphi_N\|_{L^1(\R^n)}$:
\begin{multline*}
\|\varphi_N\|_{L^1(\R^n)}= \int \limits _{{\R^n} \textbackslash{Q^N}} \frac{ldy}{(l^2+|x-y|^2)^\frac{n-1}{2}} \leq \int \limits _{|x-y|\geq \frac{N l}{2}} \frac {ldy}{(l^2+|x-y|^2)^\frac{n-1}{2}}
\\=\sum_{i=1}^\infty \int \limits_ {|x-y| \in [\frac{Nli}{2}, \frac{Nl(i+1)}{2}]} \frac{ldy}{(l^2+|x-y|^2)^\frac{n+1}{2}} \leq \sum_{i=1}^\infty \int \limits_ {|x-y| \in [\frac{Nli}{2}, \frac{Nl(i+1)}{2}]} \frac{ldy}{(l^2+\frac{N^2 l^2i^2}{4})^\frac{n+1}{2}} 
 \\ =\sum_{i=1}^\infty \biggl(\frac{C_3 l}{l^{n+1}(1+\frac {N^2i^2}{4})^\frac{n+1}{2}} \cdot \frac{N^n \cdot l^n \cdot (i+1)^n -N^n \cdot l^n \cdot i^n}{2^n} \biggr) \leq \sum_{i=1}^\infty  \frac {C_4}{N} \cdot \frac{i^{n-1}}{i^{n+1}} = \frac{C_5}{N}, 
 \end{multline*} 
where $C_5$ depends only on $n$. Further, for $C_1  2^{q-1} \|\varphi_N\|^{q-1}  $ to be less than 1, it is sufficient to require that $N$ is greater than $2 C_5 (C_1)^\frac{1}{q-1}$. We shall consider precisely such $N$.

Now it follows from the estimation~\eqref{eq1} that the expression~\eqref{eq2} does not exceed the expression~\eqref{eq3}.
\begin{equation*} 
	\begin{split} 
&\frac{1}{(N^2+1)^\frac{n+1}{2}}\int\limits_{Q^N}f_1^q(y)dy\cdot \frac{1}{|Q|}+\int \limits_{{\R^n} \textbackslash{Q^N}} f_1^q(y)\frac{ldy}{(l^2+|x-y|^2)^\frac{n+1}{2}}\leq \\
&\leq C_1 2^{q-1} \left(\int \limits_{Q^N} f_1(y) \frac{ldy}{(l^2+|x-y|^2)^\frac{n+1}{2}}\right)^q +\int \limits_{{\R^n} \textbackslash{Q^N}} f_1^q(y) \frac{ldy}{(l^2+|x-y|^2)^\frac{n+1}{2}}.
	\end{split}
\end{equation*}

Hence we have
$$\frac{1}{|Q^N|} \int \limits_{Q^N}f_1^q(y)dy \leq C_6 \left( \frac{1}{|Q^N|}  \int \limits_{Q^N}f_1(y)dy\right)^q.$$  
As one may notice, what we have just obtained is the inverse Hölder inequality for the function $f_1$ and the cube $Q^N$. Since $N$ is independent from $Q$, we can initially choose the cube $Q^{\frac{1}{N}}$ instead of $Q$. This yields the inverse Hölder inequality for the function $f_1$ and any cube $Q \in \R^n$. Hence (see [1, p.403]) it holds that $f_1 \in A_\infty$. Thanks to [1, p.409], from that we have $\log f_1 = \log (f^\epsilon) \in BMO(\R^n)$ and therefore $\log f \in BMO(\R^n).$
\end{proof}
\begin{rem}
The sufficiency in Conjecture 1 follows from the lemma that we have just proved.
\end{rem}
\begin{rem}
The condition $f \in L^p(\R^n)$ is not very important  in Lemma 2.1. Namely, $\|\log f\|_{BMO(\R^n)}$ does not depend on $p$ in the lemma's conditions. The norm mentioned above depends only on dimension of space $\R^n$ and on a constant $c$ in the inequality  $|R_jf|<cf$.
\end{rem}

 \begin{lemma}[``Necessity'']
Let $f>0$ be a measurable function on $\R^n$ such that $\log f \in \text {BMO}(\R^n)$. Then there exists $\alpha \in \R, g_1 \in L^2(\R^n), g_2 \in L^2(\R^n),  g_i >0; i \in {1,2}$ such that $f= (\frac{g_1} {g_2})^{\alpha}$ and $|{R_j g_i}| \leq c g_i$ for all $j$ from $1$ to $n$ and $i \in \{1,2\}$.
\end{lemma}
\begin{proof}
 Since $\log f \in \text {BMO}(\R^n)$, there exist $\alpha \in \R$ such that $f=v^{\alpha}$ and $v^2 \in A_2$ (see [1, p.409]). As a consequence, we get that functions $v,\frac {1}{v},\frac{1}{v^2}$ belong to the class $A_2$.

Let us consider operator $S$, defined in such way:
$$S(g)(x)=\sum_{i=1}^n |R_i (g)(x)| + \sum_{i=1}^n |R_i(gv)(x)| \cdot \frac{1}{v(x)}.$$
 Now, let us prove that $S$ is a continuous operator: 1). from $L^2(v)$ to $L^2(v)$, 2). from $L^2(\R^n)$ to  $L^2(\R^n)$, 3). from $L^2(v^2)$ to $L^2(v^2)$.

1). We can estimate a square of the norm of function  $Sg$ in $L^2(v)$ assuming that $g \in L^2(v)$.
$$\|Sg\|^2_{ L^2(v)}=\int_{\R^n} \left(\sum_{i=1}^n|R_j(g)(x)|+|R_j(gv)(x)|\frac{1}{v(x)}\right)^2 v(x) dx \leq $$
$$C\int_{\R^n}\sum_{i=1}^n|R_j(g)(x)|^2 v(x)dx+C \int_{\R^n}\sum_{i=1}^n |R_j(gv)(x)|^2\frac{1}{v(x)}dx \leq \cdots$$

Note that since $R_j$ is a continuous operator from $L^2(v)$ to $L^2(v)$ (see [1,p.411]), we have $\|R_jg\|^2_{L^2(v)}< c\|g\|^2_{L^2(v)}$. Also, as $R_j$ is continuous from $L^2(\frac{1}{v}) $ to $L^2(\frac{1}{v})$ ($\frac{1}{v}\in A^2)$, and because $gv\in L^2(\frac{1}{v})$, we have that $$\|R_j(gv)\|^2_{L^2(\frac{1}{v})}<c\|gv\|^2_{L^2(\frac{1}{v})}.$$

Now it is possible to finish a sequence of inequalities:
$$\cdots \leq C\|g\|^{2}_{L^2(v)}+ C\|gv\|^{2}_{L^2(\frac{1}{v})}=C\|g\|^{2}_{L^2(v)}.$$

2). In the same manner we can estimate norm of function $Sg$ in $L^2(\R^n) $ raised to power $2$, assuming that $g \in L^2(\R^n)$:
$$\|Sg\|^2_{ L^2(\R^n)}=\int_{\R^n} \left(\sum_{i=1}^n|R_j(g)(x)|+|R_j(gv)(x)|\frac{1}{v(x)}\right)^2 dx \leq $$
$$C\int_{\R^n}\sum_{i=1}^n|R_j(g)(x)|^2 dx+C \int_{\R^n}\sum_{i=1}^n |R_j(gv)(x)|^2\frac{1}{v(x)^2}dx \leq \cdots$$

Since $R_j$ is a continuous operator from $ L^2(\R^n)$ to $L^2(\R^n)$, then 
$$\|R_jg\|^2_{L^2(\R^n)}< c\|g\|^2_{L^2(\R^n)}$$ 

Also, because $\frac{1}{v^2}\in A_2$, and $gv\in L^2(\frac{1}{v^2})$ (since $g\in L^2(\R^n)$), we get that   
$$\|R_j(gv)\|^2_{L^2(\frac{1}{v^2})}< c\|gv\|^2_{L^2(\frac{1}{v^2}) }.$$
 Let us continue a sequence of inequalities:
$$\cdots\leq C\|g\|^{2}_{L^2(\R^n)}+ C\|gv\|^{2}_{L^2(\frac{1}{v^2})}=c_0\|g\|^{2}_{L^2(\R^n)}.$$

3). At last we will estimate norm $\|Sg\|_{L^2(v^2)}^{2}$ assuming that $g \in L^2(v^2)$:
$$\|Sg\|^2_{L^2(v)}=\int_{\R^n} \left(\sum_{i=1}^n|R_j(g)(x)|+|R_j(gv)(x)|\frac{1}{v(x)}\right)^2 v(x)^2 dx \leq $$
$$C\int_{\R^n}\sum_{i=1}^n|R_j(g)(x)|^2 v(x)^2 dx+C \int_{\R^n}\sum_{i=1}^n |R_j(gv)(x)|^2dx \leq \cdots$$
Since $R_ j$ is a continous operator from $ L^2(v^2)$ to $L^2(v^2)$($v^2\in A_2$), and $g\in L^2(v^2)$, we have that $|| R_jg||^2_{ L^2(v^2) }< c\|g\|^2_{ L^2(v^2) }$. And, because $R_ j$ is a continuous operator from $ L^2(\R^n)$ to $L^2(\R^n)$ and $gv\in L^2(\R^n)$, we also have that
 $$\| R_j(gv)\|^2_{ L^2(\R^n) }< c\| gv\|^2_{ L^2(\R^n) }.$$
By using that we finish sequence of inequalities
$$\cdots \leq C\|g\|^{2}_{L^2(v^2)}+ C\|gv\|^{2}_{L^2(\R^n)}=c_1\|g\|^{2}_{L^2(v^2)}.$$

We now return to the proof of the lemma. Let us find  function $v$ as a ratio of two functions $g_1$ and $g_2$, where $g_i\in L^2(\R^n)$ and $|R_jg_i|< c g_{i}$ for each and every $j$ from $1$ to $n$ and $i\in \{1,2\}$.

We define a sequence of functions $\{f_i\}_{i=0}^{\infty}$ as follows:
\begin{itemize}
\item $f_0$ is any function that have the following properties: $f_0>0, f_0\in L^2(R^n), f_0\in L^2(v), f_0\in L^2(v^2)$. For example, one can use the characteristic function of unit ball $B$ as such function.
\item  $f_i(x):=Sf_{i-1}(x)$ for $1 \leq i$. Note that the following inclusions hold according to the proof above:  $f_i\in L^2(\R^n), f_i\in L^2(v), f_i\in L^2(v^2)$.
\end{itemize}

Now let us define the function $g_2$ as follows:  
$$g_2(x)=\sum_{k=0}^\infty\beta ^k f_k(x),$$ where  $\beta$ is chosen so that $g_2$ would lie in $L^2(\R^n)$. That is quite possible because when $\beta \leq \frac{1}{2c_0}$, corresponding series converges absolutely:
$$\sum_{k=0}^\infty\beta ^k \|f_k(x)\|_{L^2(\R^n)} \leq \sum_{k=0}^\infty\beta^k c^k \|f_0(x)\|_{L^2(R^n)} \leq$$
$$\leq \sum_{k=0}^\infty \frac{1}{2^k} \|f_0(x)\|_{L^2(R^n)}=\|f_0(x)\|_{L^2(\R^n)}<\infty $$

Similarly, we obtain the expression for $g_1$:
$$g_2(x)=v(x)g_2(x)=\sum_{k=0}^\infty v(x)\beta ^k f_k(x).$$
Let us show that  $g_1$ belongs to space  $L^2(\R^n)$ when $\beta \leq \frac{1}{2c_1}$. We prove again that the corresponding series converges absolutely in $L^2(R^n)$:
$$\sum_{k=0}^\infty\beta ^k \|v(x) f_k(x)\|_{L^2(\R^n)} =\sum_{k=0}^\infty\beta ^k \|f_k(x)\|_{L^2(v^2)} \leq \sum_{k=0}^\infty\beta^k c_{1}^k \|f_0(x)\|_{L^2(v^2)} \leq$$
$$\leq \sum_{k=0}^\infty \frac{1}{2^k} \|f_0(x)\|_{L^2(v^2)}=\|f_0(x)\|_{L^2(v^2)}<\infty $$
Finally we prove that $g_1$ and $g_2$ are such that  $|R_j(g_i)|<cg_i$ for all $j$ from $1$ to $n$, and $i\in \{1,2\}$.
At first, let us look at $g_2$:
$$|R_j(g_2)|= \sum_{k=0}^\infty\beta ^k |R_j(f_k)|\leq \sum_{k=0}^\infty\beta ^k |S(f_k)|=$$
$$=\sum_{k=0}^\infty\beta ^k f_{k+1}<\frac{f_0}{\beta}+\frac{1}{\beta}\sum_{k=1}^\infty\beta ^k f_k=\frac{1}{\beta}\sum_{k=0}^\infty\beta ^k f_k=\frac{1}{\beta}g_2.$$
Now let us prove the same statement for $g_1$:
$$|R_j(g_1)| = |R_j(vb)| = \sum_{k=0}^\infty\beta ^k |R_j(vf_k)|\leq \sum_{k=0}^\infty\beta ^k |S(f_k)|v=$$
$$=v\sum_{k=0}^\infty\beta ^k f_{k+1}<v\left(\frac{f_0}{\beta}+\frac{1}{\beta}\sum_{k=1}^\infty\beta ^k f_k\right)=\frac{v}{\beta}\sum_{k=0}^\infty\beta ^k f_k =\frac{1}{\beta}v b=\frac{1}{\beta}g_1.$$
\end{proof}
\begin{rem}
We would like to note that similar statements to Lemma 2.2 are valid if Riesz transform are substituted with any singular integral operator  or maximal operator. The proofs of corresponding statements are similar to the proof of Lemma 2.2
\end{rem}
\begin{rem}
Similar to Lemma 2.1 statement is correct if Riesz transforms are substituted with maximal operator. In case of substitution with maximal operator the proof follows from well known facts (see [1]).
\end{rem}
\section{Proof of the main statement}
Let us remind statement of the main result of this paper  Theorem 1.
\begin{theorem1}
Let $f>0$ be a measurable function on $\R^n$. Then $\log f\in BMO(\R^n)$ if and only if there exists positive functions $g_1\in L^2(\R^n), g_2\in L^2(\R^n)$ and $\alpha \in \R$ such that  $f=(\frac{g_1}{g_2})^\alpha$ and also $|R_jg_i|\leq cg_i$ for all $j$ from $1$ to $n$, and $i\in \{1,2\}$.
\end{theorem1}
\begin{proof}
The implication ``$\Rightarrow$''. The statement is precisely Lemma 2.2.

The implication ``$\Leftarrow$''. From Lemma 2.1 we know that $\log g_i \in BMO(\R^n)$ when $i\in\{1,2\}$. At the same time it holds that 
$$\log f=\alpha \log g_1 -\alpha \log g_2 \in BMO(\R^n).$$
So, the theorem is proved.
\end{proof}

\renewcommand{\refname}{References}

\end{document}